\documentclass{amsart}
\usepackage[utf8]{inputenc}
\usepackage{proof}
\usepackage{amsfonts}
\usepackage{amsthm}
\usepackage{amssymb}
\usepackage{amsmath}
\usepackage{amstext}
\usepackage{bbm}
\usepackage{dsfont}
\usepackage{mathtools}

\usepackage{graphicx}
\usepackage{color}
\usepackage{transparent}
\usepackage{caption}
\usepackage{subcaption}
\usepackage{wrapfig}
\usepackage{color}
\newtheorem{theorem}{Theorem}[section]

\newtheorem{defi}[theorem]{Definition}

\newtheorem{stw}[theorem]{Fact}
\newtheorem{obs}[theorem]{Observation}
\newtheorem{prop}[theorem]{Proposition}

\newtheorem{lemat}[theorem]{Lemma}

\newtheorem{remark}[theorem]{Remark}
\newtheorem*{remarks}{Remark}

\newtheorem*{thmA}{Theorem A}
\newtheorem*{thmB}{Theorem B}
\newtheorem*{corC}{Corollary C}

\newcommand{\szereg}{\sum_{n=0}^{\infty} }
\newcommand{\muf}{\widehat \mu}
\newcommand{\inty}{\int_{-\infty}^{\infty}}
\newcommand{\intj}{\int_{J}}

\newcommand{\into}{\int_{\Omega}}
\newcommand{\iloczyn}{\prod_{n=0}^{\infty}}
\newcommand{\anp}{\alpha_n}
\newcommand{\gx}{g(b_nx+\theta_n)}
\newcommand{\gy}{g(b_ny+\theta_n)}

\newcommand{\aneps}{A_{n}}
\newcommand{\anoeps}{A_{n_0}}
\newcommand{\anieps}{A_{n_1}}
\newcommand{\deltag}{|\gx-\gy|}

\newcommand{\bnn}{B_{n_0,n_1} }
\newcommand{\leb}{\mathcal{L}}
\newcommand{\A}{\mathcal{A}}

\newcommand{\Z}{\mathbb{Z}}
\newcommand{\N}{\mathbb{N}}
\newcommand{\R}{\mathbb{R}}
\newcommand{\Prob}{\mathbb{P}}

\newcounter{sekcja}
\author{Julia Romanowska}%\inst{1}}
\address{Institute of Mathematics, University of Warsaw, ul. Banacha 2, 02-097 Warsaw, Poland}
\email{romanoju@mimuw.edu.pl}
\title[Measure  and Hausdorff dimension of Weierstrass-type functions]{Measure and Hausdorff dimension of randomized Weierstrass-type functions}

\begin{document}

\maketitle

\begin{abstract}
In this paper we consider functions of the type
$$f(x)  = \szereg a_n g(b_nx+\theta_n),$$ where $(a_n)$ are independent random variables uniformly distributed on $(-a^n, a^n)$ for some $0<a<1$, $b_{n+1}/b_n \geq b >1$, $a^2b> 1$ and $g$  is a $C^1$ periodic real function with finite number of critical points in every bounded interval. We prove that the occupation measure for $f$ has $L^2$ density almost surely. Furthermore, the Hausdorff dimension of the graph of $f$ is almost surely equal to $D = 2+ \log{a}/\log{b}$ provided $ b = \lim_{n\rightarrow \infty}b_{n+1}/b_n>1$ and $ab>1$.
\end{abstract}
%%%%%%%%%%%%%%%%%%%%%%%%%%%%%%%%%%%%%%%%
%INTRODUCTION
%%%%%%%%%%%%%%%%%%%%%%%%%%%%%%%%%%%%%%%%
\section{Introduction}
In this paper we study a family of nowhere differentiable functions, among which probably the most famous example is the Weierstrass function (1872):
\begin{equation*}\label{weierstrass}
W(x)= \szereg a^n \cos(2 \pi b^nx).
\end{equation*}
 Weierstrass proved that the function is nowhere differentiable for some class of $a$ and $b$, later Hardy (\cite{hardy1916}) extended the result for all $a, b$ such that $0< a< 1< b$ and $ab> 1$.
Functions of the Weierstrass type were considered by  Besicovitch and Ursell (\cite{besicovitchursell}) in 1930s and later in 1980s by Berry and Lewis (\cite{berry1980weierstrass}) and Ledrappier (\cite{ledrappier1992dimension}) as examples of fractal curves, questions about dimension were raised. As the graph of $W(x)$ is self-affine in the sense that $aW(bx)$ differs from  $W(x)$ by a smooth function $\cos(2\pi x)$ it suggests that the dimension should be equal to $D = 2 - \alpha$ for $\alpha=-\frac{\log a}{\log b}$ (notice that under the previous conditions, $ 1< D< 2$). Kaplan, Mallet-Paret and Yorke (\cite{kaplan1984lyapunov}) in 1984 proved that box-counting dimension is equal to $D$. However, the question of determining the Hausdorff dimension of the graph is still not completely solved.
Przytycki and Urbański (\cite{przytycki1989hausdorff}) in 1989 proved that the Hausdorff dimension of the graph is bigger than 1. Mauldin and Williams (\cite{mauldin1986hausdorff}) in  1986 considered a function $w_b (x) = \sum_{-\infty}^\infty b^{-\alpha n} \left[ \phi (b^nx+\theta_n) - \phi (\theta_n)\right]$ for $b>1$, $0<\alpha <1$, arbitrary $\theta_n$ and $\phi$ with period 1. They proved that for sufficiently large $b$ the Hausdorff dimension of the graph of $w_b$ has a lower bound $2-\alpha - \frac{C}{\ln b}$. 

Recently, Biacino in \cite{biacino} %considered functions of the form $\szereg \frac{\varphi(b_n x)}{b_n^\alpha}$ for $0<\alpha<1$, $\varphi: \R\rightarrow \R$ smooth, $(b_n)_{n=0}^\infty$ such that there exists $B>1$ and $\delta>0$ such that $B\leq \frac{b_{n+1}}{b_n}\leq \frac{1}{\delta}$ and
 showed that if $b$ is large enough then the Hausdorff dimension of the graph of $W$ is equal to $2-\alpha$.

 In \cite{hunt} Hunt  considered function of the form
\begin{equation*}
H(x) = \szereg a^n \cos(2 \pi b^n x + \theta_n)
\end{equation*}
 with random phases $\theta_n$  (independent random variables with the same distribution). Using potential theory methods he proved that the Hausdorff dimension of the graph is almost surely equal to $D$. Other dynamical systems with random phases were studied by Kifer (\cite{kifer1997computations}). 
Many papers, such as e.g. \cite{kono1986self},  support the hypothesis that when $h$ is roughly self-affine and $\mu$ is absolutely continuous with respect to Lebesgue measure Hausdorff and box-counting dimension coincide. However, if  $\frac{b_{n+1}}{b_n}\rightarrow \infty$, then in many cases Hausdorff dimension is strictly smaller than the upper box-counting dimension, see \cite{baranski2012dimension,besicovitchursell}.

 In this paper we perturb randomly the parameter $a$ in the Weierstrass type functions and obtain that almost surely Hausdorff dimension of the graph is $D$. The result is formulated as the following theorem:
 \begin{thmA}\label{theoremA}

Assume that $f(x) = \sum_{n=0}^\infty a_n g(b_nx+\theta_n)$ satisfies the following conditions:
\begin{enumerate}
\item $\left(a_n\right)_{n=0}^\infty$ is a sequence of  real independent random variables defined on some probabilistic space $(\Omega, \Prob)$ with uniform distribution on $(-a^n, a^n)$   for some $0<a<1$,
\item $\lim_{n\rightarrow \infty} \frac{b_{n+1}}{b_n} = b$ for some $b>1$, $ab> 1$,
\item   $\theta_n \in \R$ for $n \in \N$.
\item $g: \R \rightarrow \R$ is $C^1$ periodic and has  a finite number of critical points in every bounded interval.
\end{enumerate}
Then the Hausdorff and box dimension of the graph of $f$ are equal to:
\begin{equation*}
\dim_H \textnormal{graph} f = D = 2+\frac{\log a }{\log b}
\end{equation*}
almost surely.

 \end{thmA}

\begin{remarks}
The condition \textnormal{(4)} is satisfied if $g$ is non-constant periodic analytic.
\end{remarks}

We also examine the occupation measure for $f$, that is 
$$ \mu(S) = \leb\left(\{x \in J: \ f(x) \in S\}\right),$$ where $J=[0,T]$, $T$ is a period of $f$ and  $\leb$ is the Lebesgue measure. 
In this paper we show that the occupation measure for the Weierstrass type function of the form
\begin{equation*}
f(x)= \szereg a_n g(b_nx+\theta_n)
\end{equation*}
for randomly chosen $a_n$ has $L^2$ density with respect to Lebesgue measure almost surely, which is stated as Theorem B:
 \begin{thmB}\label{theoremB}
 Let $f(x)= \szereg a_n g(b_nx+\theta_n)$ for $g: \R \rightarrow \R$, $T$-periodic, $C^1$ with a finite number of critical points in every bounded interval,  satisfy the following conditions:
 \begin{enumerate}
\item  $(a_n)_{n=0}^\infty$ is a sequence of independent random variables defined on some probabilistic space $(\Omega, \Prob)$ with uniform distribution on $(-a^n, a^n)$, $0<a<1$, 
\item $(b_n)_{n=0}^\infty$, there exists $b>1$ such that  $\frac{b_{n+1}}{b_n}\geq b$ for all $n>0$.
\item $a^2b> 1$, \label{ass3}
\item $\theta_n \in \R$ .
\end{enumerate}
Then the occupation measure for the function $f$ is absolutely continuous with $L^2$ density almost surely. 

Moreover, if $b_n=b^n$, $b \in \N, \ b>1$ and $\theta_n=0$ for every $n$, then the assumption \eqref{ass3} may be replaced by $ab> 1$. 
 \end{thmB}
 A result of this kind (Theorem B) was announced in (\cite{hunt}), but to our knowledge it has never been published.

\begin{corC}
For the Weierstrass-type function of the form 
\begin{equation*}
w_a(x) = \szereg a_n \cos (2\pi b^n x)
\end{equation*}
 if
\begin{enumerate}
\item $(a_n)_{n=0}^\infty$ is a sequence of independent random variables with uniform distribution on $(-a^n, a^n)$, $0<a<1$, 
\item $b \in \N$ and $ab> 1$, $b>1$
\end{enumerate}
then almost surely the Hausdorff dimension of the graph is equal to $D= 2+\frac{\log a}{\log b}$ and the occupation measure is absolutely continuous with $L^2$ density.

\end{corC}

The paper is organized as follows: in Section 2 we present basic notation and discussion on the assumptions which should be made on the function $g$. Section 3 provides proof of Theorem A,  in Section 4 we state Theorem B, in the following Section 5 some technical lemmas are proved, and  finally Section 6 discusses an example - the Weierstrass function.

%%%%%%%%%%%%%%%%%%%%%%%%%%%%%%%%%%%%%%%%
%SEKCJA 2
%%%%%%%%%%%%%%%%%%%%%%%%%%%%%%%%%%%%%%%%
\section{Preliminaries}

For basic definitions and properties of the Hausdorff dimension, we refer to books by Falconer \cite{falconer2007fractal} and Mattila \cite{mattila1999geometry}. By $\leb$ we denote an appropriate Lebesgue measure (on $\R$ or $\R^2$) and for a given set $A$ we denote its complement  by $A^c$. The Hausdorff dimension and box dimension are denoted respectively as  $\dim_H$, $\dim_B$.

Now we will present some consequences of the assumptions made on $g$.
%We would like to prove that measure $\mu$ has $L^2$ density almost surely. To do that, we need to make some assumptions on how should the function $g$ look like.
%SEKCJA 3
\newcommand{\E}{\mathbb{E}}
\begin{lemat}\label{properties_g}
Let $g: \R \rightarrow \R$ be a  periodic $C^1$ function of period $1$ with a finite number of critical points in $[0,1]$. Then there exists $C>0$  such that for all $\epsilon>0$ there is $\delta>0$ such that $\delta \xrightarrow[]{\epsilon\rightarrow 0} 0$ and one can cover the set
\begin{equation*}
A = \left\{(x,y) \in [0,1]^2: \left|g(x)-g(y)\right|<\epsilon\right\}
\end{equation*}
  with $N \leq \frac{C}{\delta}$ squares with vertical and horizontal sides of length $\delta$.
\end{lemat}

\begin{proof}
Let $m$ be the number of critical points of $g$ in  $[0,1]$.
Fix $\epsilon>0$.
%, $x \in [0,1]$ and consider the set $E=\{ t: |g(x)-g(t)|<\epsilon\}$. 
Since $g\in C^1$, for all $\rho>0$  there exists $\delta(\rho)>0$ such that $\delta(\rho)\xrightarrow[\rho\rightarrow 0]{}0$ and the set $\left\{x\in [0,1]: |g'(x)|<\rho\right\}$ can be covered by $m$ intervals $I_1, \dots I_m$ of length $\delta(\rho)$.
Since $\delta(\rho) \rightarrow 0$, there exists $\rho_\epsilon >0 $ such that 
\begin{equation}\label{rhodelta}
\epsilon \leq \rho_\epsilon \delta(\rho_\epsilon)
\end{equation}

Set $\rho = \rho_\epsilon$, $\delta = \delta(\rho_\epsilon)$.
 %, $I_i$:  $C_\rho \subset \bigcup_{i=1}^m I_i$.
%We can cover the rest of the interval $[0,1]$ by intervals $J_j$, such that $[0,1]^2\setminus \bigcup_{i=1}^m I_i = \bigcup_{i=1}^{m\pm 1}J_j$.
Let $J_j \subseteq [0,1]$, $j = 1,\dots, M$, $M\in \{ m-1, m, m+1\}$,  be the gaps between intervals $I_i$.
It is obvious that the set 
\begin{equation} \label{cover}
\bigcup_{i=1}^m \left(I_i\times[0,1] \cup [0,1] \times I_i \right)
\end{equation} can be covered by $\frac{C_1}{\delta}$ squares of side $\delta$, for some constant $C_1$ independent of $\epsilon$.

Now, take $j, k \leq M$. Suppose $g'|_{J_j}\geq \rho$ and $g'|_{J_k}\geq\rho$ (the cases when $g$ is decreasing on $J_j$ or $J_k$ can be proved analogously). 
%$x\in J_j$ such that $g'(x) > \rho$ (so that $g$ on $J_j$ is increasing, analogously for decreasing). 
%Consider a set
By the definition of $A$, if $A \cap\left(J_j \times J_{k}\right) \neq \emptyset$ then 
\begin{equation*}
A \cap\left(J_j \times J_{k}\right) \subset
\left\{ (x,y): x \in J_j, \ \inf g(J_k) - \epsilon \leq g(x) \leq \sup g(J_k) + \epsilon, \  h_1(x) <y < h_2(x) \right\}
\end{equation*}
%= \left\{ (x,y): x \in J_j, g^{-1}(g(x)-\epsilon) < y < g^{-1}(g(x)+\epsilon) \right\}\\
where $$h_1(x) = (g|_{J_k})^{-1} (\max\{g(x)-\epsilon, \inf g(J_k) \})$$ 
 and $$h_2(x) = (g|_{J_k})^{-1}( \min\{g(x)+\epsilon,\sup g(J_k)\})$$
 are defined on some  interval in $J_j$ and are continuous and nondecreasing. It is easy to check that the graph of $h_1$ can be covered by $\frac{C_2\left| J_j\right|}{\delta}$ squares of side $\delta$. 
  By the Mean Value Theorem and \eqref{rhodelta} we have $\left| h_1(x) - h_2(x) \right| \leq \frac{2\epsilon}{\rho} \leq 2\delta $.
  Thus we obtain that $ A \cap \left(J_j \times J_{k }\right) $ can be covered by   $\frac{C_3 \left| J_j\right|}{\delta}$ squares of side $\delta$. Summing over $j$ and $k$ we obtain that the set $[0,1]^2 \setminus \bigcup_{i=1}^m \left(I_i\times[0,1] \cup [0,1] \times I_i \right)$ can be covered by $\frac{C_4}{\delta}$ squares of side $\delta$. Together with estimations for the set  \eqref{cover} we conclude the proof. 
\end{proof}

\begin{defi}\label{22} For $A \subset [0,1]^2$ and $ \underline{ \Theta} = (\underline{\theta_1}, \underline{\theta_2}, \dots)$, where $\underline{\theta_j} \in \R^2,$ we define:
\begin{enumerate}
\item  $\A = \sum_{n, m \in \Z} (A + (n,m))$,
\item $\A_n (\underline{\Theta})=[0,1]^2 \cap  \A \cap \left(\frac{\A -\underline{\theta_1}}{b_1} \right) \cap \left(\frac{\A-\underline{\theta_2}}{b_2} \right) \cap \dots \cap \left(\frac{\A-\underline{\theta_n}}{b_n} \right)$, where $$\frac{\A-\underline{\theta_j}}{b_j} =\left\{(x,y): (b_jx, b_j y) + \underline{\theta_j}  \in \A \right\} $$
\end{enumerate}
\end{defi}

Now we will state an important geometric lemma.
\begin{lemat}\label{kwadraciki}
Let $(b_n)_{n=0}^\infty$ such that $b_n>0$, $b_0=1$ and $\frac{b_{n+1}}{b_n}\geq b$ for some $b>1$. Fix $C>0$.
Suppose that $A\subset [0,1]^2$ such that for some $\delta>0$ the set $A$ can be covered by $N\leq \frac{C}{\delta}$ squares of vertical and horizontal sides of length $\delta$.

Then for sufficiently small $\delta$ there exists $\tilde{C}>0$ such that for every $\underline{\Theta}$ and every $n>0$
\begin{equation}\label{oszmiara}
\leb(\A_n(\underline{\Theta})) < \tilde{C}\gamma^n,
\end{equation}
where  $0<\gamma<\frac{1}{b}+\epsilon_0<1$ and  $\epsilon_0 = \epsilon_0(\delta)>0$  is arbitrarily small if $\delta$ is small enough. 

\end{lemat}

\begin{proof}
\begin{figure}[lh]
\centering
  \begin{subfigure}[b]{0.5 \textwidth}
\def\svgwidth{180pt}
 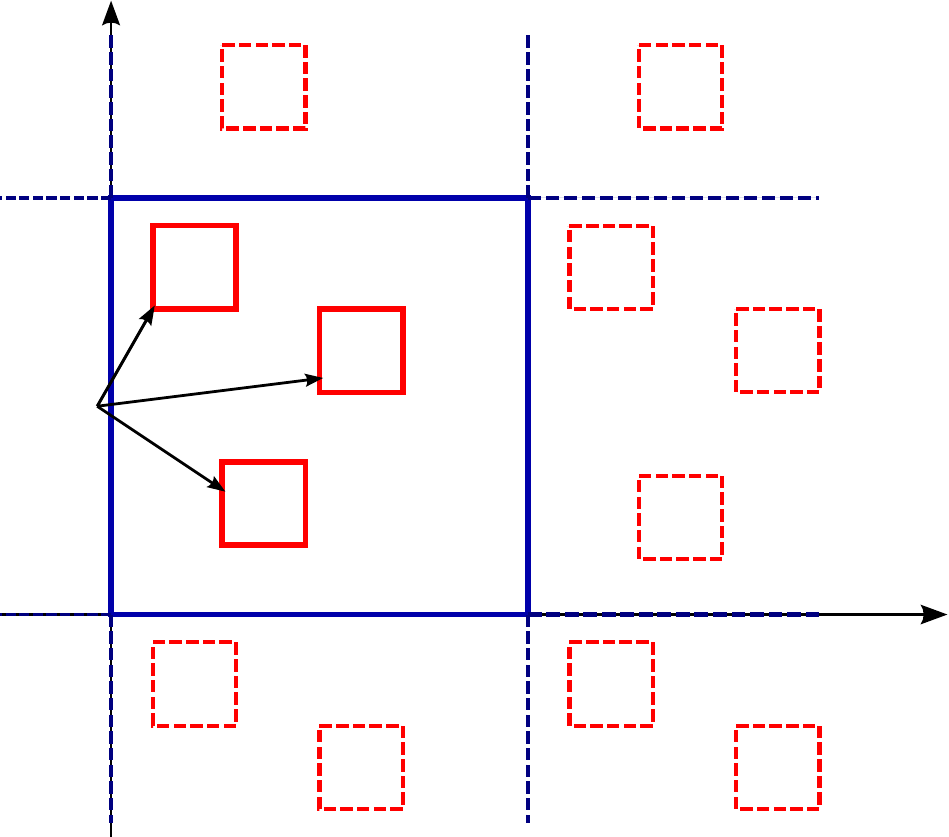
 \caption{The set $\A$}
 \label{setA}
\end{subfigure}
\qquad\qquad\qquad
\begin{subfigure}[b]{0.8\textwidth}
  \def\svgwidth{300pt}
  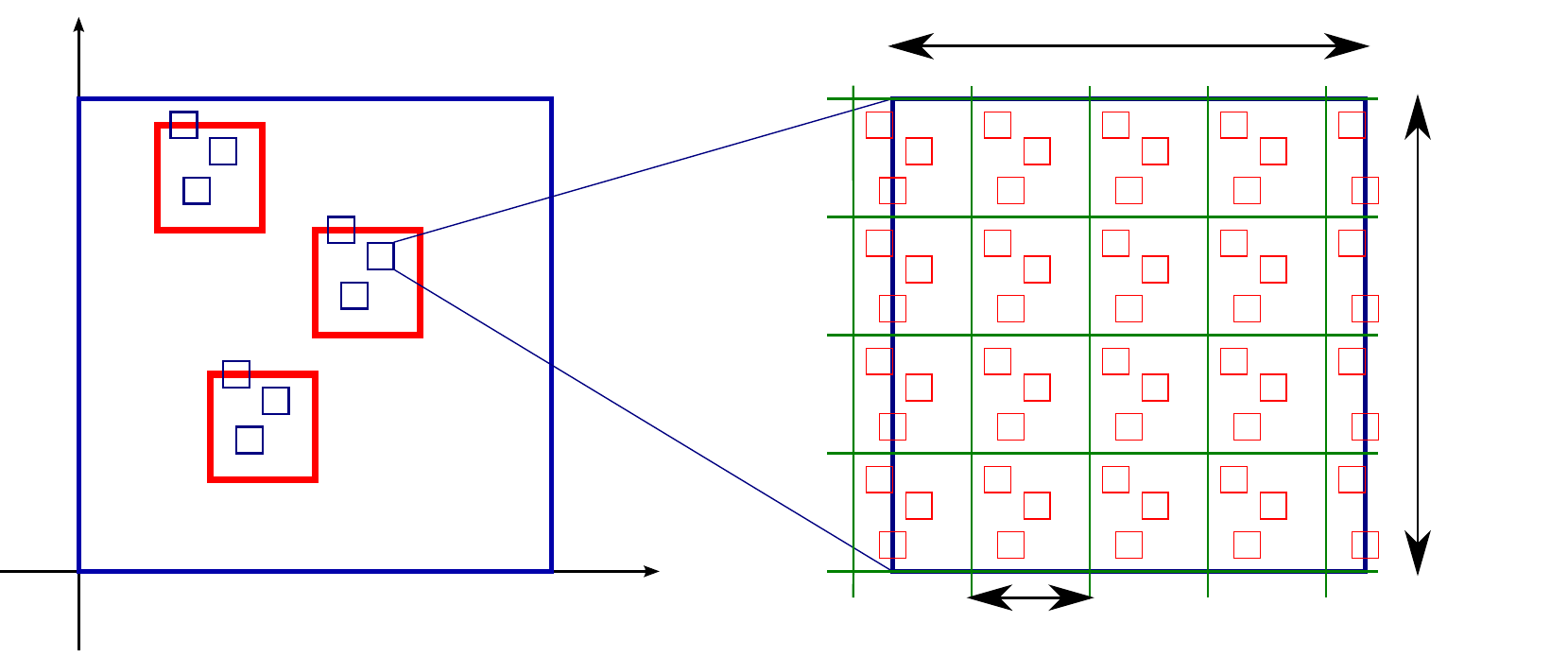
  \caption{One of squares from $n$-th step with $\frac{1}{b_{n+1}}$ grid moved by $\theta_{n+1}$ and a copy of covering of $\A$ inside}
   \label{zoomA}

  \end{subfigure}
\end{figure}
For sufficiently small $\delta$ we can take $k>0$ such that

\begin{equation}\label{wark}
\frac{2}{\delta b^k} < 1 \leq \frac{2}{\delta b^{k-1}}
\end{equation}
and let us take
\begin{equation*}\label{defgamma}
\gamma = \sqrt[k]{N\left(\delta+\frac{2}{b^k}\right)^2}.
\end{equation*}
By \eqref{wark} we can estimate $\gamma$:
\begin{equation*}
\gamma<\sqrt[k]{C\delta\left(1+\frac{2}{\delta b^k}\right)^2}
< \sqrt[k]{4C\delta}
\end{equation*}
and as $b^{k-1} \leq \frac{2}{\delta}$:
\begin{equation*}
k \leq \frac{\log\frac{2}{\delta}}{\log b}+1 = \frac{\log\frac{2}{\delta}+\log b}{\log b}.
\end{equation*}
Thus,
\begin{equation*}
\gamma < (4C \delta)^{\frac{\log b}{\log\frac{2}{\delta}+\log b}} \xrightarrow[\delta\rightarrow 0]{}\frac{1}{b}.
\end{equation*}
Hence, 
\begin{equation*}
\gamma < \frac{1}{b} + \epsilon_0<1.
\end{equation*}
with $\epsilon_0$ arbitrarily small for  sufficiently small $\delta$. For simplicity we write $\A_n = \A_n(\underline{\Theta})$.
Take $m= 0,1,\dots$. Let $N_0 = N$, $\delta_0 = \delta$. Let $\delta_m =\frac{\delta}{b_{mk}}$. Proceeding by induction on $m$ we construct a sequence of coverings of $\A_{mk}$ by squares of horizontal and vertical sides of length $\delta_m$.
  Note that $\A_0 = A$ can be covered by $N_0 = N$ squares of side length $\delta_0=\delta$.
Suppose that $\A_{mk}$ can be covered by $N_m$ squares of side length $\delta_m$ for some $N_m$. 
We estimate $\frac{\delta_m}{\delta_{m+1}}$. Between $m$th and $(m+1)$th steps we shrink our set by $\frac{b_{mk+1}}{b_{mk}}, \frac{b_{mk+2}}{b_{mk+1}}, \dots, \frac{b_{mk+k}}{b_{mk+k-1}}$, each of the ratios is larger or equal than $b$, by the assumptions,  thus

\begin{equation*}\label{delty}
\frac{\delta_m}{\delta_{m+1}} = \frac{b_{mk+1}}{b_{mk}}\cdots \frac{b_{mk+k}}{b_{mk+k-1}}  \geq b^k.
\end{equation*}

%Therefore, our set will be shrunk by at least $b^{k}$ and moved by  $\underline{\theta}_{mk+1}, \underline{\theta}_{mk+2}, \dots, \underline{\theta}_{mk+k}$.

 To  calculate how many copies of shrunk set $\A_{mk}$ we get in $[0,1]^2$, so in $\A_{(m+1)k}$, we will cover it with squares with sides $\delta_{m+1}$ and calculate their number $N_{m+1}$.  To do it easily we will 
cover the plane with a square grid with side $\frac{1}{b^{k}}$ and translate the grid by $\underline{\theta}_{mk+1}, \underline{\theta}_{mk+2}, \dots, \underline{\theta}_{mk+k}$.  Each square $Q$ of side $\delta_m$ is now divided into new squares from the grid, some of them possibly sticking out of $Q$. The total number of grid squares in each row can be estimated by:
	\begin{itemize}
	\item $\frac{\delta}{\frac{1}{b^{k}}} = \delta b^{k}$ - squares which are completely inside $Q$ in each row,
	\item there may be at most 2 squares which stick out of $Q$ (horizontally).
	\end{itemize}
So we obtain $\delta b^{k} +2$ grid squares in each row. As we may have at most $\delta b^{k} + 2$ rows (again, 2 rows may stick out vertically of $Q$), the total number of new generation of  squares in $Q$ is at most $N (\delta b^{k} + 2)^2$ (set $\A$ is covered by $N$ squares). We have $N_{m}$ different squares $Q$ of side $\delta_m$, hence the number $N_{m+1}$ of squares of sides $\delta_{m+1}$ covering $\A_{(m+1)k}$ satisfies:
\begin{equation*}
N_{m+1} \leq N_{m} \left(b^{k}\delta+2\right)^2 N.
\end{equation*}

Now, let $\leb_{m} = N_{m} \delta_m^2$.  We get a measure ratio:
\begin{eqnarray*}
\frac{\leb_{m+1}}{\leb_{m}} &= &\frac{N_{m+1}\left(\delta_{m+1}\right)^2}{N_{m} \left(\delta_m\right)^2}\notag\\
& \leq & \frac{N_{m} \left[N(\delta b^{k} +2 )^2 \right]\left(\delta_{m+1}\right)^2}{N_{m} \left(\delta_m\right)^2}  \notag\\
& = & N(\delta b^{k} +2 )^2\left(\frac{1}{b^{k}}\right)^2 \notag\\
& = & N\left(\delta +\frac{2}{b^k}\right)^2 = \gamma^k
\end{eqnarray*}
%&= & N\delta^2\left(1+\frac{2}{b^{k}\delta}\right)^2= \alpha <1 \notag\label{bdelta2}.
%= \epsilon \left(1+\frac{2}{b^k\delta}\right)^2
%Furthermore, estimating the measure of $A_{mk}$:
This implies
\begin{equation*}
\leb(A_{mk}) \leq \leb_{m} \leq \leb_0 \gamma^{mk} < \gamma^{mk}
\end{equation*}
Now, let $n \in \N$, take $m$ such that  $mk \leq n < (m+1)k$. Then 
\begin{equation*}
\leb(A_n) \leq \leb(A_{mk}) < \gamma^{mk}< \gamma^{n-k} = \tilde{C} \gamma^n,
\end{equation*}
for $\tilde{C} = \frac{1}{\gamma^k}$.
Thus the proof is finished.
\end{proof}

\begin{remark}\label{druga_czesc_wynikanie} 
In Lemma \ref{kwadraciki} if $b_n=b^n$ for $b \in \N, \ b >1$ and $\theta_n = 0$ for every $n$, then  in \eqref{oszmiara} we can take $\gamma = N \delta^2$.%$\gamma = \leb(A)$.
\end{remark}
\begin{proof}
If $b_n=b^n$ for $b \in \N, \ b >1$ and $\theta_n = 0$ then squares from the grid do not stick out of $Q$, %we only have squares inside $Q$, 
so their number in $(m+1)$th step satisfies:
 \begin{equation*}
 N_{m+1} \leq N_m \left(\delta b^k\right)^2  N
 \end{equation*}
 and thus 
 \begin{equation*}
 \frac{\leb_{m+1}}{\leb_m} \leq N \delta^2
 \end{equation*}
 which is exactly the measure of the covering of set $A$.
\end{proof}
Remark \ref{druga_czesc_wynikanie} will be used in the proof of the second part of Theorem B.

%SEKCJA 4 

%%%%%%%%%%%%%%%%%%%%%%%%%%%%%%%%%%%%%%%%
%HAUSDORFF DIMENSION
%%%%%%%%%%%%%%%%%%%%%%%%%%%%%%%%%%%%%%%%
\section{Hausdorff dimension}
In this section we will prove, similarly to \cite{hunt} that the Hausdorff dimension of the graph of $f$ is equal to $D = 2+\frac{\log a}{\log b}$. 
  \begin{thmA}\label{theoremA}

Assume that $f(x) = \sum_{n=0}^\infty a_n g(b_nx+\theta_n)$ satisfies the following conditions:
\begin{enumerate}
\item $\left(a_n\right)_{n=0}^\infty$ is a sequence of  real independent random variables defined on some probabilistic space $(\Omega, \Prob)$ with uniform distribution on $(-a^n, a^n)$   for some $0<a<1$,
\item $\lim_{n\rightarrow \infty} \frac{b_{n+1}}{b_n} = b$ for some $b>1$, $ab> 1$,
\item   $\theta_n \in \R$ for $n \in \N$.
\item $g: \R \rightarrow \R$ is $C^1$ periodic and has  a finite number of critical points in every bounded interval.
\end{enumerate}
Then the Hausdorff and box dimension of the graph of $f$ are equal to:
\begin{equation*}
\dim_H \textnormal{graph} f = D = 2+\frac{\log a }{\log b}
\end{equation*}
almost surely.

 \end{thmA}

We can obviously assume that the period of $g$ is 1. In the estimations we consider the graph of $f$ over the interval $J=[0,1]$. 
\subsection{Upper bound}
We would like to calculate the lower box dimension of the graph of the function, which from the definition is:
\begin{equation*}
\underline{\dim_B} \ \textnormal{graph}(f) = \varliminf_{\epsilon\rightarrow 0} \frac{\log N(\epsilon)}{-\log\epsilon} 
\end{equation*}
where $N(\epsilon)$ denotes the minimal number of balls of radius $\epsilon$ which cover our set.
Fix $\epsilon >0$, let $n$ be the minimal number such that $\frac{1}{b_n}<\epsilon$. We will estimate the number $N(\frac{1}{b_n})$.
Let us divide the interval $J$ into intervals of length $\frac{1}{b_n}$ (the last interval may be shorter) and denote one of such intervals as $I$. Fix $x, y \in I$. We have $\left|x-y\right| \leq \frac{1}{b_n}$ and we obtain
\begin{equation}\label{gora}
\left| f(x) - f(y) \right| \leq L \left(\left|a_0\right| b_0 + \dots + \left| a_n \right|b_n\right) \left|x-y\right| + 2M \sum_{k=n+1}^\infty |a_k|
\end{equation}
where $L$ is the Lipschitz constant of $g$ and $M= \sup_{x\in J} |g(x)|$.
As $|a_k| \leq a^k$ we get 
\begin{equation}\label{ak}
\sum_{k=n+1}^\infty |a_k| \leq \sum_{k=n+1}^\infty a^k  = \frac{a^{n+1}}{1-a} %<\frac{1}{1-a}.
\end{equation}
Fix  $\frac{1}{a} <b'<b$. Then $\frac{b_{n+1}}{b_n}>b'$ for every $n\geq n_0$ for some $n_0$ and
\begin{eqnarray*}
|a_0|b_0+\dots + |a_n|b_n \leq b_0 + ab_1+\dots +a^n b_n \leq c+ \frac{a^{n_0+1}b_n}{(b')^{n-n_0-1}}+\dots +\frac{a^{n-1}b_n}{b'}+a^nb_n\\
= c +a^n b_n \left(\frac{1}{(ab')^{n-n_0-1}}+\dots+ \frac{1}{ab'}+1\right) < c+ a^nb_n \frac{1}{1-\frac{1}{ab'}}
\end{eqnarray*}
%\frac{b_n}{b^n} +\dots + \frac{a^{n-1}b_n}{b^n} + a^nb_n,
where $c=  b_0 + ab_1+\dots +a^{n_0} b_{n_0}$.
%Thus,
%\begin{equation*}
%a_0b_0+\dots a_nb_n \leq  a^n b_n \left(\frac{1}{a^nb^n} +\dots +\frac{1}{ab}+1\right) = a^n b_n \frac{1-\frac{1}{(ab)^n}}{1-\frac{1}{ab}} < a^n b_n  \frac{1}{1-\frac{1}{ab}}
%\end{equation*}
Using this together with  \eqref{ak} and \eqref{gora} we obtain:
\begin{equation*}
\left| f(x) - f(y) \right| \leq \frac{Lc}{b_n}+\left( \frac{ab'}{ab'-1} + \frac{2Ma}{1-a}\right)a^n %\leq  \left(\frac{L}{1-ab} + M\frac{1}{1-a}\right)a^nb_n = \tilde{L} a^n b_n
\end{equation*}

We have $b_n > c_1(b')^n> \frac{c_1}{a^n}$ for some $c_1$. 
Thus
\begin{equation*}
\left| f(x) - f(y)\right| \leq c_2 a^n.
\end{equation*}
Since we have at most $b_n+1$ intervals $I$, 
\begin{equation*}
N\left(\frac{1}{b_n}\right) = c_2 a^n b_n (b_n+1) = c_3a^n b_n^2.
\end{equation*}

Therefore, 
\begin{equation*}
\underline{\dim_B} \textnormal{graph} f = \varliminf_{n\rightarrow\infty}\frac{\log N\left(\frac{1}{b_n}\right)}{-\log\left(\frac{1}{b_n}\right)} 
=\varliminf_{n\rightarrow\infty}\frac{\log c_3a^n b_n^2}{\log b_n} \leq 2 + \lim_{n\rightarrow \infty} \frac{n\log a}{\log b_n} = 2 + \frac{\log a}{\log b}
\end{equation*}
which, as $\dim_H \textnormal{graph} f \leq \underline{\dim_B} \textnormal{graph} f $ concludes the proof of this case.

\subsection{Lower bound}
In the proof we follow a method used in \cite{hunt}.
We will use potential theory methods and estimates of the  $t$-energy of the measure $\nu$, which by definition is equal to
\newcommand{\itt}{I_t}
\begin{equation}\label{tenergia}
\itt(\nu) = \iint_{J \times J} \frac{d\nu(x)d\nu(y)}{|x-y|^t}.
\end{equation}
 As $\dim_H(A) = \inf \{t : \itt(\nu)<\infty \textnormal{ for some measure }\nu\textnormal{ supported on } A\}$, if we show that  \eqref{tenergia} is finite for some $t$, we obtain that the Hausdorff dimension is greater than $t$. Choosing a sequence of $t$s approaching $D$, we will get our result --- $\dim_H\textnormal{graph} f \geq D$.

Let $\nu$ be the Lebesgue measure lifted to the graph of $f$. We obtain:
 
 \begin{equation}\label{energiam}
 \itt(\nu) = \iint_{J\times J} \frac{dxdy}{\left((x-y)^2+(f(x)-f(y))^2\right)^\frac{t}{2}}
 \end{equation}
 
 Let us fix $t\in (1,D)$. To show that \eqref{energiam} is finite almost surely we will show that 
 \begin{equation}\label{et}
E_t =  \int_\Omega \itt(\nu) d\Prob= \int_\Omega \iint_{J\times J} \frac{dxdy}{\left((x-y)^2+(f(x)-f(y))^2\right)^\frac{t}{2}}d\Prob
 \end{equation}
 
 is finite. By the Fubini theorem:
 \begin{equation*}
E_t =  \intj\intj \int_\Omega  \frac{1}{\left((x-y)^2+(f(x)-f(y))^2\right)^\frac{t}{2}} d\Prob dx dy
\end{equation*}
 
 %We would like to show that  for $x, y\in J$
 %\begin{equation*}
 %\int_\Omega  \frac{1}{\left((x-y)^2+(f(x)-f(y))^2\right)^\frac{t}{2}} d\Prob \leq C \sup_z h(z) \left| x-y \right|^{D-t-1}
 %\end{equation*}
 
% and since $t< D$ it will follow that $E_t  <\infty$.
 
 Now, let $z_{x,y} = f(x) - f(y)$ for some $x, y\in J$. 
 As $z_{x,y}$ is a sum of independent random variables we may write its density $h_{x,y}$ as an infinite convolution  $h_{x,y} = h^{(0)}_{x,y} \ast h^{(1)}_{x,y} \ast \dots$ of densities:
 \begin{equation*}
h^{(n)}_{x,y} = \frac{\mathbbm{1}_{\left[-a^n\left|g(b_n x +\theta_n) -g(b_n y +\theta_n)\right|, a^n\left|g(b_n x +\theta_n) -g(b_n y +\theta_n)\right|\right]}}{2a^n\left|g(b_n x +\theta_n) -g(b_n y +\theta_n)\right|}.
\end{equation*}
 %As $z$ is a random variable, we will show that it has a bounded density. Thus
 Furthermore,

 \begin{eqnarray}\label{et2}
 \int_\Omega  \frac{d\Prob}{\left((x-y)^2+(f(x)-f(y))^2\right)^\frac{t}{2}} & =&    \inty \frac{h_{x,y}(s)ds}{\left((x-y)^2+s^2\right)^\frac{t}{2}}\\
& =& \inty \frac{h_{x,y}(\left|x-y\right|w)|x-y|dw}{\left| x-y \right|^t \left(1+w^2\right)^\frac{t}{2}} \notag\\
&\leq & C  \frac{\sup h_{x,y}}{|x-y|^{t-1}} \notag
%&\leq & \sup_z h(z) \left|x-y \right|^{1-t} \inty \frac{dw}{\left(1+w^2\right)^\frac{t}{2}}.
 \end{eqnarray}
for some $C>0$, because $t>1$.
Fix $\epsilon>0$. 
\begin{defi}\label{zbioran} For $n\geq 0$ define  the sets
$$\aneps = \{(x,y)\in J \times J: \deltag \geq\epsilon\}$$
\end{defi}

Let us define the set $B_n$:
\begin{equation*}
 B_n= A_0^c \cap A_1^c \cap \dots \cap A_{n-1}^c \cap A_n.
 \end{equation*}
 
 We can see that $J \times J =  \bigcup_{n} B_n \cup C$, where $C = \left(J \times J\right) \setminus \bigcup_{n}B_n$.
% We will show that $C$ is a small set.

 Take a small $\epsilon > 0$ and set $A = \left(J\times J\right) \cap A_0^c$, $\underline{\theta_n} = \left(\theta_n, \theta_n\right)$.  Then the set $\A_n (\underline{\Theta})$ from the Definition \ref{22} is equal to $A_0^c \cap \dots \cap \aneps^c$. Applying Lemma \ref{properties_g} and Lemma \ref{kwadraciki} we obtain 
 \begin{equation*}
 \leb( A_0^c \cap \dots \cap \aneps^c) < \tilde{C} \gamma^n
 \end{equation*}
 for $\gamma < \frac{1}{b} + \epsilon_0$ where $\epsilon_0$ is arbitrarily small for small $\epsilon$. 
 Since $C = \left(J \times J \right)\cap \bigcup_{n=0}^\infty A_n^c$ this implies the following lemma:
 \begin{lemat}\label{32}
$ \leb(C) = 0.$
 \end{lemat} 

Moreover, 
\begin{equation} \label{oszbn}
\leb(B_n) \leq \leb(A_0^c\cap A_1^c \cap \dots \cap A_{n-1}^c)< \tilde{C}  \gamma^{n-1}.
\end{equation}

Take $(x, y) \in B_n$. We have 
\begin{equation*}
\epsilon \leq \left| g(b_nx+\theta_n) - g(b_ny+\theta_n) \right| {\leq} Lb_n\left| x-y \right|
\end{equation*}
 where $L$ is a Lipschitz constant of $g$. 
Since $h_{x,y} $ is the convolution  of $h^{(n)}_{x,y}$ we have
\begin{equation*}
\sup_{B_n} h_{x,y} \leq \sup h^{(n)}_{x,y}\leq \frac{1}{2a^n\epsilon}
\end{equation*}
On the other hand, taking $b' > b$ arbitrarily close to $b$, we obtain $b_n \leq c (b')^n$ for some $c>0$ and

\begin{equation*}
|x-y|^{1-t} \leq \left(\frac{\epsilon}{Lb_n}\right)^{1-t} \leq \left(\frac{\epsilon}{L(b')^n}\right)^{1-t}
\end{equation*}

%Take some $b' > b$ then $b_n \leq c (b')^n$ for some $c>0$. 
%and from the assumptions $\frac{b_{n+1}}{b_n} \rightarrow b$, so that for sufficiently large $n$  there exists $b'<b$ such that $\frac{b_{n+1}}{b_n}\geq  b'$. 
 
By this and \eqref{et} and \eqref{et2},
 
 \begin{eqnarray*}
 E_t
 &\leq &  C \sum_n \int_{B_n} \left| x-y\right|^{1-t} \sup h^{(n)}_{x,y} dxdy\\
 &\leq &  C \sum_n \int_{B_n} \left(\frac{\epsilon}{Lc (b')^n}\right)^{1-t} \frac{1}{2a^n \epsilon}dxdy = C_1 \sum_n \int_{B_n}\frac{1}{a^n (b')^{n(1-t)}}dxdy\\
 &\leq & C_2 \sum_n \frac{\leb(B_n)}{(a(b')^{1-t})^n}\leq C_3 \sum_n \left(\frac{\gamma}{a(b')^{1-t}}\right)^n <\infty
 \end{eqnarray*}
 if only $\gamma < a(b')^{1-t}$. The last step is to check this condition.
 
Since $\gamma < \frac{1}{b}+\epsilon_0 $ for arbitrarily small $\epsilon_0$ and $b'$ can be chosen arbitrarily close to $b$ it is sufficient to check that 
\begin{equation}\label{ostgamma}
\frac{1}{b}< ab^{1-t}
\end{equation}
which holds because $t < D$. Hence we obtain the finiteness of \eqref{et}, which concludes this case.
%%%%%%%%%%%%%%%%%%%%%%%%%%%%%%%
%SEKCJA 5
%%%%%%%%%%%%%%%%%%%%%%%%%%%%%%%
\section{Proof of Theorem B}
We state {\bf Theorem B} once again:

  \begin{thmB}\label{theoremB}
 Let $f(x)= \szereg a_n g(b_nx+\theta_n)$ for $g: \R \rightarrow \R$, $T$-periodic, $C^1$ with a finite number of critical points in every bounded interval,  satisfy the following conditions:
 \begin{enumerate}
\item  $(a_n)_{n=0}^\infty$ is a sequence of independent random variables defined on some probabilistic space $(\Omega, \Prob)$ with uniform distribution on $(-a^n, a^n)$, $0<a<1$, 
\item $(b_n)_{n=0}^\infty$, there exists $b>1$ such that  $\frac{b_{n+1}}{b_n}\geq b$ for all $n>0$.
\item $a^2b> 1$, \label{ass3}
\item $\theta_n \in \R$ .
\end{enumerate}
Then the occupation measure for the function $f$ is absolutely continuous with $L^2$ density almost surely. 

Moreover, if $b_n=b^n$, $b \in \N, \ b>1$ and $\theta_n=0$ for every $n$, then the assumption \eqref{ass3} may be replaced by $ab> 1$. 
 \end{thmB}
 \begin{proof}
 In the  the proof we will use methods used by  Peres and Solomyak, see e. g. \cite{peres1996absolute}. 

We would like to prove that $||\mu||_2<\infty$ almost surely. By the Parseval formula it is sufficient to prove that $||\muf||_2 < \infty$ almost surely, where $\muf$ is the Fourier transform of the measure $\mu$. 
As previously we can assume $T=1$ and consider the graph over the  interval $J=[0,1]$.

The Fourier transform of $\mu$ is defined as
$$\muf (u) = \inty e^{iut}d\mu(t) =\int_0^1 e^{iuf(x)}dx,$$
for $u \in \R$. We have
\begin{eqnarray}\label{definicja_muf}
||\muf||_2^2 & = &\inty|\muf(u)|^2du = \inty \muf(u)\overline{\muf(u)}du=\inty\intj e^{iuf(x)}dx \intj e^{-iuf(y)}dydu\notag\\ 
&=&  \inty\intj\intj e^{iu(f(x)-f(y))}dxdydu
\end{eqnarray}
We will integrate this expression over the probabilistic space $\Omega$. 
\begin{eqnarray*}\label{definicja_i}
I&=&\into||\muf||_2^2d\Prob=\into\inty|\muf(u)|^2dud\Prob = \into\inty\intj\intj e^{iu(f(x)-f(y))}dxdydu\Prob\\
& = &\lim_{u_0\rightarrow \infty} \into \int_{-u_0}^{u_0}\intj\intj e^{iu(f(x)-f(y))}dxdydu
=\lim_{u_0\rightarrow \infty} I_{u_0}
\end{eqnarray*}

If $I$ is finite, the integral \eqref{definicja_muf} (so that our norm) is also finite almost surely. 

Using the Fubini theorem in $I_{u_0}$ we may change the integration order and get:
$$I_{u_0}= \int_{-u_0}^{u_0}\intj\intj\into e^{iu(f(x)-f(y))}d\Prob dxdydu.$$

Let us denote  $$Z_n=a_n(g(b_nx+\theta_n)-g(b_ny+\theta_n)),$$
 where $(Z_n)_{n=0}^\infty$ - independent random variables with uniform distribution on  $(-\alpha_n, \alpha_n)$ for 
 $$\alpha_n=a^n(g(b_nx+\theta_n)-g(b_ny+\theta_n)).$$ 
 
 Since $f(x)$ is a series of independent random variables we obtain
  %, the inner integral is a characteristic function. Characteristic function of a sum of independent random variables is a product of its characteristic functions, thus
\begin{eqnarray*}
I_{u_0}  & = & \int_{-u_0}^{u_0}\intj\intj \into e^{iu(f(x)-f(y))}d\Prob dxdydu  = \int_{-u_0}^{u_0}\intj\intj\iloczyn\into e^{iuZ_n}d\Prob dxdydu \\
& = & \int_{-u_0}^{u_0}\intj\intj\iloczyn \int_{-\alpha_n}^{\alpha_n}\frac{1}{2\alpha_n} e^{iut}dtdxdydu 
= \int_{-u_0}^{u_0}\intj\intj\iloczyn \frac{e^{iu\anp}-e^{-iu\anp}}{ 2i\alpha_n u}dxdydu\\
& =& \int_{-u_0}^{u_0} \intj\intj \iloczyn \frac{\sin(u\anp)}{u\anp}dx dy du .
\end{eqnarray*}
%where we obtain last equation by integrating the inner function.

Now, let us denote 
$$x_n=\anp u = a^n(g(b_nx+\theta_n)-g(b_ny+\theta_n))u.$$
Then
\begin{equation}\label{sinxn}
I_{u_0} = \int_{-u_0}^{u_0} \intj\intj \iloczyn \frac{\sin(x_n)}{x_n}dx dy du 
\end{equation}
To conclude the proof it is sufficient to show that the integral \eqref{sinxn} is finite and the estimations are independent from $u_0$. This will be done in the following proposition.
%t \ref{main_fact} from which finiteness of the \eqref{sinxn} integral will result. 
\begin{prop}\label{main_fact}
Let $x_n = a^n(g(b_nx+\theta_n)-g(b_ny+\theta_n))u$ for some $u \in \R$, where $x,y \in [0,1]$ and $a, b_n, g$ are defined in Theorem B.
% $(b_n)_{n=0}^\infty$, $(a_n)_{n=0}^\infty$ - independent random variables,  $a_n$ with uniform distribution on $(-a^n, a^n)$ and $g$ defined as in the statement of Theorem B. 
Then
\begin{equation*}
 \inty \intj\intj \iloczyn \left|\frac{\sin(x_n)}{x_n}\right| dx dy du < \infty.
  \end{equation*}
\end{prop}
%From the finiteness of \eqref{sinxn} we obtain that the occupation measure has $L^2$ density almost surely, which ends the proof of Theorem B.
%In next sections we prove theorem \ref{main_fact}, from which follows that \eqref{sinxn} is finite so that it will conclude the proof.
\end{proof}

\section{Proof of Proposition \ref{main_fact}}
%In this section in few steps we will prove Proposition \ref{main_fact}.

%\begin{proof} (Of Proposition \ref{main_fact}) %TODO
%\begin{proof} \textit{(Beginning of the proof of Proposition \ref{main_fact})}
Fix $M>0$. We can divide our integral into three parts, which will be estimated separately.
\begin{eqnarray*}
& & \inty \intj\intj \iloczyn \left|\frac{\sin(x_n)}{x_n}\right|dx dy du \\
&=&  \int_{-M}^{M}\intj\intj \iloczyn \left|\frac{\sin(x_n)}{x_n}\right|dx dy du 
 + \int_{-\infty}^{-M}\intj\intj \iloczyn\left| \frac{\sin(x_n)}{x_n}\right|dx dy du\\
& + & \int_M^{\infty}\intj\intj \iloczyn \left|\frac{\sin(x_n)}{x_n}\right|dx dy du  \\
 & = & I_1+I_2+I_3\\
\end{eqnarray*}
Using the fact that  $\left|\frac{\sin x}{x} \right| \leq 1$:
 \begin{eqnarray*}
\left|I_1\right| & \leq &  \int_{-M}^{M}\intj\intj \iloczyn \left|\frac{\sin(x_n)}{x_n}\right|dx dy du \leq  \leb (J^2)  \int_{-M}^M 1 du\\
&\leq &\leb(J^2) 2M <\infty
\end{eqnarray*}

As $I_2$ and $I_3$ can be estimated in the same way, we will estimate only $I_2$.

Fix $\epsilon>0$.
Consider the set
$$\aneps = \{(x,y)\in J \times J: \deltag \geq \epsilon\}$$
as in Definition \ref{zbioran}.

\begin{stw} \label{3.1}
\begin{equation*}
 \iint_{\anoeps\cap \anieps}\iloczyn \left|\frac{\sin x_n}{x_n}\right| dxdy
\leq \frac{\leb(\anoeps \cap \anieps)}{a^{n_0}a^{n_1} u^2 \epsilon^2}, \ \textnormal{ for } n_0\neq n_1\geq 0 
\end{equation*}
\end{stw}

\begin{proof}
\begin{eqnarray*}
 \iint_{\anoeps \cap \anieps} \ \iloczyn \ \left| \frac{\sin x_n}{x_n} \right| dx dy 
 =  \iint_{\anoeps\cap \anieps} \left| \frac{\sin x_{n_0}}{x_{n_0}} \frac{\sin x_{n_1}}{x_{n_1}} \right| \prod_{n \in \mathbb{N}\setminus\{n_0,n_1\} } \left|\frac{\sin x_n}{x_n}\right|dx dy \\
\end{eqnarray*}
Since we can estimate  $\prod_{n \in \mathbb{N}\setminus\{n_0,n_1\} } \left|\frac{\sin x_n}{x_n} \right|\leq 1$ and $\left|\sin x_{n_{i}}\right|\leq 1$, $i = 0,1$ and use the definition of the set $\anoeps \cap \anieps$,
we obtain ${\displaystyle \left|\frac{\sin x_{n_{i}}}{x_{n_{i}}}\right| \leq \left|\frac{1}{\epsilon u a^{n_{i}}}\right|}$, so that 
\begin{eqnarray*}
\lefteqn{\iint_{\anoeps\cap \anieps} \left| \frac{\sin x_{n_0}}{x_{n_0}} \frac{\sin x_{n_1}}{x_{n_1}}\right| \prod_{n \in \mathbb{N}\setminus\{n_0,n_1\} } \left|\frac{\sin x_n}{x_n} \right|dxdy  }\\
& \leq  &  \iint_{\anoeps\cap \anieps} \left| \frac{1}{\epsilon^2 u^2 a^{n_0+n_1}}\right|dx dy \\
&= & \frac{\leb(\anoeps\cap \anieps)}{\epsilon^2 u^2 a^{n_0+n_1}} \\
\end{eqnarray*}
\end{proof}

\begin {defi} For $0 \leq n_0 < n_1$ define
$$\bnn= A_{0}^c\cap  A_{1}^c\cap  \dots \cap A_{n_0-1}^c\cap A_{n_0} \cap A_{n_0+1}^c \cap \dots \cap  A_{n_1-1}^c \cap A_{n_1}$$
\end{defi}
It means that $\bnn$ is a set in which the condition  $|g(b_nx+\theta_n) - g(b_ny+\theta_n)|{\geq}\epsilon$ hold the first time for $n_0$ and the next time for $n_1$.  % $|\deltago|>\epsilon$ oraz $|\deltagi|>\epsilon$.}

Now we divide $J \times J = \bigcup_{n_0} \bigcup_{n_1> n_0} \bnn \cup C$, where $C = \left(J \times J\right) \setminus \bigcup_{n_0} \bigcup_{n_1> n_0}\bnn$.
We would like to prove the following lemma:
\begin{lemat} $\leb(C)=0$. 
\end{lemat}

\begin{proof}
The proof is similar to the proof of Lemma \ref{32} and is left to the reader.
% We can easily see that:
%\begin{eqnarray*}\label{aneps}
%\lefteqn{ \aneps = \{(x,y): \ \deltag >\epsilon \}  } \\
%&= & \{(\frac{x-\theta_n}{b_n},\frac{y-\theta_n}{b_n}): \ |g(x)-g(y)| >\epsilon \} \\
%& = & \{ (\frac{x}{b_n},\frac{y}{b_n}): \ |g(x)-g(y)| >\epsilon\} \\
%& = & \frac{1}{b_n}  \{ (x,y): |g(x)-g(y)| > \epsilon \} \\
%& = & \frac{1}{b_n} A_{0}
%\end{eqnarray*}
%Hence, $C = \bigcap_{n=0}^\infty \aneps^c= \bigcap_{n=0}^\infty \frac{1}{b_n} A_{0}^c = A_0^c \iloczyn \frac{1}{b_n}$.
%Since $b_n>1$ we obtain $\leb(C) \leq  \leb(A_0^c) \iloczyn \frac{1}{b_n} < \tilde{\epsilon}$.
\end{proof}

\begin{obs} $$\iint_{\bnn} \iloczyn \left|    \frac{\sin x_n}{x_n}  \right|dx dy  \leq \frac{\leb(\bnn)}{\epsilon^2 u^2 a^{n_0}a^{n_1}}$$
\end{obs}
\begin{proof}
It easily follows from fact \ref{3.1}.
\end{proof}

\begin{obs} $$\iint _{J \times J}   \iloczyn \left|    \frac{\sin x_n}{x_n} \right|dx dy
 \leq \frac{1}{\epsilon^2u^2}\sum_{n_0} \sum_{n_1> n_0} \frac{\mathcal{L}(\bnn)}{a^{n_0}a^{n_1}}$$
\end{obs}
\begin{proof}
	\begin{eqnarray*}
	\lefteqn{\iint _{J \times J}   \iloczyn \left|   \frac{\sin x_n}{x_n}\right|    dx dy= \iint_{\bigcup_{n_0}\bigcup_{n_1> n_0}\bnn \cup C} \iloczyn \left|    \frac{\sin x_n}{x_n}\right|   dx dy} \\
	& = &    \iint _{\bigcup_{n_0}\bigcup_{n_1> n_0} \bnn} \iloczyn \left|    \frac{\sin x_n}{x_n}\right|   dx dy\\
 	& \leq & \sum_{n_0}\sum_{n_1 > n_0} \iint _{\bnn} \iloczyn \left| \frac{\sin x_n}{x_n}\right| dx dy \\
	& \leq & \sum_{n_0}\sum_{n_1 > n_0} \frac{\leb(\bnn)}{\epsilon^2 u^2 a^{n_0} a^{n_1}} 
	\end{eqnarray*}
	%\textnormal{ The third equality holds because $\bnn$ are disjoint sets and previous observation implies the last inequality.}
\end{proof}

To complete the proof we now need to show the following lemma:
\begin{lemat}\label{glowny}
%For function satisfying conditions $1-4$ from the statement of Theorem~B and $\bnn$ defined as above we obtain:
We have
\begin{equation}\label{szereg_bn}
\sum_{n_0}\sum_{n_1 > n_0} \frac{\leb(\bnn)}{a^{n_0}a^{n_1}} < \infty. 
\end{equation}
\end{lemat}

%We will now prove lemma used in proof of lemma  \ref{glowny}

%\begin{ass} There exists $b>1$ such that for all $n>0$ $\frac{b_{n+1}}{b_n}\geq b$. $b_0=1$.
%\end{ass}

\begin{proof}%(\textit{of lemma \ref{glowny}})\\
Fix $0 \leq n_0 < n_1$.  We need to estimate the measures $\leb(\bnn)$.
By definition,  $\bnn=  A_{0}^c\cap  A_{1}^c\cap  \dots \cap A_{n_0-1}^c\cap A_{n_0} \cap A_{n_0+1}^c \cap \dots \cap  A_{n_1-1}^c \cap A_{n_1}$. We obtain:

\begin{equation}\label{seta}
%\begin{split}
\leb(\bnn)\leq \leb(A_{0}^c\cap  A_{1}^c\cap  \dots \cap A_{n_0-1}^c \cap 
A_{n_0+1}^c\cap  A_{n_0+2}^c\cap  \dots \cap A_{n_1-1}^c)
%\end{split}
\end{equation}

Now set  $A =\left(J \times J\right) \cap A_0^c$, 
$$\tilde{b}_n = \begin{cases}
b_n & \textnormal{ for }  n = 0, \dots, n_0-1\\
b_{n+1} &  \textnormal{ for }  n \geq n_0%, \dots, n_1-2
\end{cases}$$
where {$\frac{\tilde{b}_{n+1}}{\tilde{b}_n}\geq b$} and 
$$\tilde{\theta}_n = \begin{cases}
\theta_n & \textnormal{ for }  n = 0, \dots, n_0-1\\
\theta_{n+1} &  \textnormal{ for }  n \geq n_0.%, \dots, n_1-2.
\end{cases}$$

Then the set $\A_{n_1-2}(\underline(\Theta))$ from Definition \ref{22} is equal to $A_{0}^c\cap  A_{1}^c\cap  \dots \cap A_{n_0-1}^c \cap A_{n_0+1}^c\cap  A_{n_0+2}^c\cap  \dots \cap A_{n_1-1}^c$.
%As we can see, we need to estimate the $\leb(A_{0}^c\cap  A_{1}^c\cap  \dots \cap A_{n}^c)$ for each $n$.%, for which useful will be the following lemma.
We may apply Lemma  \ref{properties_g} and Lemma \ref{kwadraciki} to obtain
%Due to properties of the function $g$ and Lemma \ref{properties_g}, we may apply Lemma \ref{kwadraciki} to the sets $A_{0}^c\cap  A_{1}^c\cap  \dots \cap A_{n_0-1}^c$ and $A_{n_0+1}^c\cap  A_{n_0+2}^c\cap  \dots \cap A_{n_1-1}^c)$ to obtain, that 
\begin{equation*}\label{szac1}
\leb(A_{0}^c\cap  A_{1}^c\cap  \dots \cap A_{n_0-1}^c \cap A_{n_0+1}^c\cap  A_{n_0+2}^c\cap  \dots \cap A_{n_1-1}^c)\leq \tilde{C} \gamma^{n_1-2}
\end{equation*}
for $\gamma < \frac{1}{b} + \epsilon_0$, where $\epsilon_0$ can be arbitrarily small.
From this and  \eqref{seta} we get:
\begin{equation}\label{last}
\begin{split}
\sum_{n_0}\sum_{n_1 > n_0} \frac{\leb(\bnn)}{a^{n_0} a^{n_1}} \leq 
 \tilde{C} \sum_{n_0}\sum_{n_1 > n_0} \frac{\gamma^{n_1-n_0+ n_0-2}}{a^{n_0+n_1}} \\
= \frac{\tilde{C}}{\gamma^2}\sum_{n_0}\frac{\gamma^{n_0}}{a^{2n_0}}\sum_{n_1-n_0> 0}\frac{\gamma^{n_1-n_0}}{a^{n_1-n_0}} = \frac{\tilde{C}}{\gamma^2 (1-\gamma)(1-\frac{\gamma}{a^2})}<\infty
\end{split}
\end{equation}
as $\frac{\gamma}{a^2}< \frac{1}{a^2}\frac{1}{b}+\frac{\epsilon_0}{a^2}<1$, because $a^2b>1$.
So that \eqref{szereg_bn} is satisfied and Lemma \ref{glowny} is proved.
\end{proof}

\begin{remark}
By Remark \ref{druga_czesc_wynikanie}, if $b_n = b^n, \ b \in \N, b>1$ and $\theta_n=0$ for every $n\in \N$ instead of $a^2b > 1$ it is sufficient to have  $ab>1$.

%we don't have to have assumption on $a$ and $b$ like  in statement 3 of Theorem  B. 
%To obtain finiteness of \eqref{szereg_bn} it is sufficient to have $ab>1$. 
This proves the second part of Theorem B.
\end{remark}
\begin{proof}
As $\gamma$ can be arbitrarily small, we can take $\gamma < \frac{a}{b}$, hence $\frac{\gamma}{a^2} < \frac{1}{ab} < 1$ and the inequality \eqref{last} holds.
\end{proof}

\begin{proof}[Proof of Proposition \ref{main_fact}]

Summing up all previous lemmas we obtain:
\begin{equation}
\begin{split}
I \leq  I_1+I_2+I_3 \leq \left|I_1\right| + 2\left| I_2\right| \leq\\
 \leb(J^2) 2M + 2 \frac{\tilde{C}}{\gamma^2 (1-\gamma)(1-\frac{\gamma}{a^2})\epsilon^2}\left|\int_M^\infty  \frac{1}{u^2}du\right|
<\infty
\end{split}
\end{equation}
which concludes the proof of Proposition \ref{main_fact} and thus the proof of both parts of Theorem B.
\end{proof}

\section{Example - the Weierstrass function}

Let us consider the function% $g$ almost like in Hunt's paper
$$W(x) =\szereg a_n \cos(2\pi b^n x)$$
where $(a_n)_{n=0}^\infty$ - independent random variables with uniform distribution on $(-a^n, a^n)$, for $0<a<1<b$, $ab>1$ and $b\in\N$. Here $g(x)=\cos(2\pi x)$.

The sets $A_n^c$ have the form:
\begin{eqnarray*}\label{anc}
 A_n^c & = & \left\{ (x,y):  \left| \cos (2\pi b^n x) - \cos (2\pi b^n y) \right| < \epsilon \right\}  =   \left\{ (\frac{x}{2\pi b^n}, \frac{y}{2\pi b^n}):  \left| \cos x -\cos y \right|< \epsilon \right\} \\
 & = &\frac{1}{2\pi b^n} \left\{(x,y): \left| 2 \cos(\frac{x+y}{2}) \cos(\frac{x-y}{2})\right|<\epsilon \right\} \notag\\
 &= & \frac{1}{2\pi b^n}\left\{(u+v,u-v):\  \left|2\cos u \cos v\right| < \epsilon \right\}\notag
\end{eqnarray*}

\begin{figure}[lh]
\centering
%\def\svgwidth{180pt}
% \includegraphics[width=\textwidth]{prog/sin/iteracje/a1b8e0_05.pdf}
 
%\resizebox{0.7\columnwidth}{!}{
\input{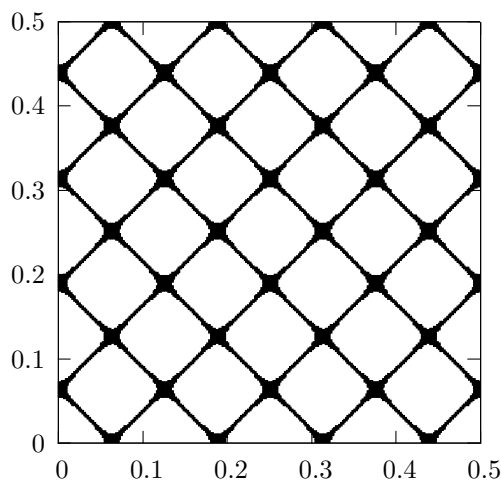}%}
 \caption{Set $A_n^c$ for $g=\cos$, $a=0.8$, $b=2$, $N=2$ and $\epsilon = 0.05$.}
 \label{zbiorA}
\end{figure}
%tutaj rysunek

The inequality $|\cos u \cos v| < \frac{\epsilon}{2}$ is true when both $|\cos u | \leq \sqrt{\frac{\epsilon}{2}}$ and $|\cos v | \leq \sqrt{\frac{\epsilon}{2}}$ (both $|\cos u| \leq 1$ and $\epsilon \leq 1$). Cosine near its zeros behaves nearly like linear function, so we can approximate the set $A_n^c$ by a sum of rectangles with width at most $C\sqrt{\epsilon}$ for constant $C>0$. It is illustrated in Figure \ref{zbiorA}.

From the second part of Theorem B we obtain that the occupation measure on the graph of $W(x)$ has $L^2$ density almost surely and from Theorem A we obtain that the Hausdorff dimension of the graph is almost surely equal to $D = 2+\frac{\log a }{\log b}$.

\bibliography{mybib}
\bibliographystyle{amsplain}

\end{document}